\documentclass[11pt,a4paper]{amsart}
\usepackage[utf8]{inputenc}
\usepackage[T1]{fontenc}
\usepackage[english]{babel}
\usepackage{amsmath}
\usepackage{amsfonts}
\usepackage{amssymb}
\usepackage{graphicx}
\usepackage{lmodern}
\usepackage[ left=3cm,right=3cm,top=3.5cm,bottom=5cm]{geometry} 
\usepackage{amsthm}
\usepackage{tikz}
\usepackage{verbatim}
\usepackage{todonotes}

\theoremstyle{plain}
\newtheorem{theorem}{Theorem}

\theoremstyle{definition}
\newtheorem{definition}{Definition}

\newtheorem{problem}{Problem}
\theoremstyle{remark}
\newtheorem{remark}{Remark}

\newcommand{\norm}[1]{\left\lVert#1\right\rVert}
\newcommand{\R}{\mathbb R}
\newcommand{\con}{u}
\newcommand{\U}{\mathcal U}
\newcommand{\Uad}{\mathcal U_\text{ad}}
\newcommand{\J}{J}
\newcommand{\C}{\mathcal C}
\newcommand\T{\rule{0pt}{2.6ex}}       % Top strut
\newcommand\B{\rule[-1.2ex]{0pt}{0pt}} % Bottom strut

\title[]{Optimal control for interacting particle systems driven by neural networks} 

\author[S.~G\"ottlich]{Simone G\"ottlich}
\address[S.~G\"ottlich]{University of Mannheim}
\email[]{goettlich@uni-mannheim.de}

\author[C.~Totzeck]{Claudia Totzeck}
\address[C.~Totzeck]{University of Mannheim}
\email[]{totzeck@uni-mannheim.de}

\date{\today}

\begin{document}

\begin{abstract}
We propose a neural network approach to model general interaction dynamics and an adjoint based stochastic gradient descent algorithm to calibrate its parameters. The parameter calibration problem is considered 
as optimal control problem that is investigated from a theoretical and numerical point of view.
We prove the existence of optimal controls, derive the corresponding first order optimality system and formulate a stochastic gradient descent algorithm to identify parameters for given data sets. 
To validate the approach we use real data sets from traffic and crowd dynamics to fit the parameters. The results are compared to forces corresponding to well-known interaction models such as the Lighthill-Whitham-Richards model for traffic and the social force model for crowd motion.
\end{abstract}

\maketitle

\smallskip
	
%%%%%%%%%%%%%%%%%%%%%%%%%%%%%%%%%%%%%%%%%
\section{Introduction}
%%%%%%%%%%%%%%%%%%%%%%%%%%%%%%%%%%%%%%%%%

In the recent years many models for interaction dynamics with various applications such as swarming, sheep and dogs, crowd motion, traffic and opinion dynamics have been proposed, see e.g.~\cite{AlbiPareschi,Schafe2,Schafe1,CarrilloSwarming,CrisPicTos2014,CuckerSmale,PareschiToscani2013,toscani2006kinetic,CollisionAvoidance} for an overview. Typically, the models are based on ordinary differential equations (ODEs)
which describe the dynamics of each particle (or agent) in the system by interaction forces. 
A common approach is to model forces that replicate observations made in nature. For example in swarming, there exists the well-known three-phase model that consists of a short range repulsion, a neutral/comfort zone and a third zone with attraction for long range interactions \cite{CarrilloSwarming}. The ranges can be adjusted with the help of parameters that need to be fitted for every application. Certainly, this is a smart approach and leads to promising results. Nevertheless, it is interesting to overboard all those assumptions and instead use a neural network to model the interactions with parameters based on data. 

Clearly, the idea of using neural networks as substitutes for static/dynamic models or observers is well-known, see e.g~\cite{Chen2018,Haber_2017} in the general context of ODEs. Furthermore, in \cite{Brakes} neural networks are discussed as alternative to estimate friction in automotive brakes. See \cite{Energy} for a survey of artificial neural networks in energy systems. Moreover, in \cite{Gruene1} deep neural networks are used to approximate Lyapunov functionals for systems of ordinary differential equations and in \cite{Gruene2} the same author proposes to store approximate Lyapunov functions in a deep neural network in order to overcome the curse of dimensionality.

We intend to follow a similar approach for interacting particle systems in this paper. We first propose a  framework to model very general particle interactions with the help of neural networks. Then we state the corresponding parameter calibration problem. It enables us to identify the parameters using techniques from optimal control. We prove the well-posedness of the identification problem and derive the corresponding first order optimality conditions that are used to compute the gradient for the stochastic descent algorithm. 

In addition to the neural network model, the proposed parameter calibration approach can be used for general interacting particle systems with explicitly given, differentiable interaction forces, for example in terms of gradients of a potential. For our experiments, we apply the stochastic descent algorithm to real  data sets from traffic and crowd experiments \cite{dataTraffic,PedestrianData} to fit the neural network parameters. For comparison we use the same algorithm to fit as well as parameters of well-known interaction models, such as the Lighthill-Whitham-Richards (LWR) model for traffic \cite{LWR_HR} and the social force model for crowd dynamics \cite{HelbingMolnar1998}.
The two applications discussed here are prototypical examples for a first order approach in 1d and second order approach in 2d, respectively. Moreover, in both cases real data is available for our calibration purposes.

Similar parameter identification studies for pedestrian models haven been recently introduced in \cite{Corbetta2015,GomesStuartWolfram} using a Bayesian probabilistic method and in \cite{GoettKnapp,Tordeux2019} using neural networks. In contrast to \cite{GoettKnapp,Tordeux2019}, where the pedestrian speed or unknown interaction forces have been estimated, we address a more general setting that also allows for theoretical investigations and a rigorous numerical treatment. More precisely, for the neural networks we make some basic assumptions, but it is important to note that we do not prescribe any physical interaction assumptions. Using techniques from optimal control, we derive and analyze a parameter identification procedure that is based on stochastic gradient descent in Section~\ref{sec:optprob}. In Section~\ref{sec:anaopt}, we begin with well-posedness results concerning a general interacting particle system that is driven by a artificial neural network. Then we prove the existence of an optimal control for the parameter identification problem. Moreover, in Section~\ref{stochgrad}, we derive the first-order optimality system for the identification problem in order to state the corresponding algorithm, an adjoint-based stochastic gradient descent method. A proof of existence of the adjoint concludes the theory part. 

To validate our approach, we present an extended parameter estimation study for the traffic and the pedestrian model in Section~\ref{traffic} and \ref{crowd}, respectively. We apply the algorithm to a traffic scenario and estimate the interaction force as well as the speed of cars. The second application is based on pedestrian data, here again we train the artificial neural network with the help of our algorithm and compare the results to optimal parameters resulting from interactions that involve the social model for pedestrian interaction. The numerical results for both applications offer interesting insights and gives rise for future considerations.

%%%%%%%%%%%%%%%%%%%%%%%%%%%%%%%%%%%%%%%%%%%%%%%%%%%%%%%%%%%
\section{Optimal Control problem}\label{sec:optprob}
%%%%%%%%%%%%%%%%%%%%%%%%%%%%%%%%%%%%%%%%%%%%%%%%%%%%%%%%%%%

The parameter identification is cast as optimal control problem. 
Let $\con$ denote the control variables, or, in terms of the application, the parameters to be identified and $z$ the reference data set. We denote the space of controls by $\U = \R^K.$ Then we are interested in 
\[	
\min\limits_{\con \in \U} \J(x(\con);z)
\]
for some tailored cost function $\J$. 
As the parameters enter the cost function implicitly through the state $x,$ we propose to compute the gradient of the cost functional used for a stochastic gradient decent method with the help of an adjoint-based approach.  
 
From now on, it is clear that $x$ depends on the parameters $\con$, so for notational convenience we drop the dependence in some equations.
Let us assume that for each cost evaluation, we consider $N$ agents with corresponding trajectories $z_i \colon [0,T] \rightarrow \R^d$ for $i=1,\dots,N.$ Based on the applications we have in mind, we obtain parameter dependent trajectories $x_i^\con \colon [0,T] \rightarrow \R^d$ for $i=1,\dots,N.$  We focus on the cost functional
$$ \J(x(\con);z) = \frac12 \int_0^T \norm{z(t) - x(t) }^2 dt = \frac12 \norm{z - x(\con)}_{L^2(0,T)}^2. $$

The state $x = x(\con)$ is a solution to an ODE system that is driven by a feed-forward artificial neural network (NN) as follows
\begin{equation}\label{ODE}
	\frac{d}{dt} x_i = \sum_{j=1}^N W^{i,j}_\con(x_j -x_i), \quad x_i(0) = z_0^i,\quad i=1,\dots,N,
\end{equation}	
where $W_\con$ models the interaction of the agents and is the output of a neural network parametrized by $\con.$

\begin{remark}
Throughout the work, we assume that every agent is driven by the same artificial neural network. We need the index $W^{i,j}$ only for the applications. For example in the traffic dynamic, the cars only interact with the car in front. This leads to $W^{i,j}$ being nonzero only for $j = i+1.$ In contrast, the car in front drives with fixed velocity, this can be represented with fixed weights, not involved in the parameter identification. Moreover, pedestrians interact with every other person which yields $W^{i,j} = W.$ 
\end{remark}

To summarize, the optimal control problem we propose for the parameter identification driven by neural networks reads
\begin{problem}\label{optProb}
	Find $\bar \con \in \Uad := [-1,1]^K$ such that 
	$$ \J(x(\bar \con),z) = \min_{\con \in \U} \J(x(\con),z) \quad \text{subject to \eqref{ODE}.}  $$
\end{problem}

\begin{remark}\label{rem:costfunctional}
	We emphasize that in contrast to \cite{GoettKnapp} the cost functional is not of the usual structure given by $$ \frac1m \left( \sum_{i=1}^m \J(h_\con(x_i),y_i) \right),  $$
	where $h_\con(x_i)$ denotes the output of a neural network defined by the parameters $\con.$ Indeed, in the present article, the trajectories $z_i(\con), i=1,\dots,N$ are not the output of the neural network, but solutions of ODEs that are driven by the neural network. 
\end{remark}

As mentioned above, we consider feed-forward artificial neural networks. For later reference we define these as follows.
\begin{definition}\label{def:NN}
A \textit{feed-forward artificial neural network (NN)} is characterized by the following structure:
\begin{itemize}
	\item [-] Input layer: $$a_1^{(1)} = 1, \quad a_k^{(1)} = x_{k-1},\; \text{ for }k\in \{ 2,\dots, n(1)+1\},$$
	where $x \in \R^{n^{(1)}}$ is the input (feature) and $n^{(1)})$ is the number of neurons without the bias unit $a_1(1).$
	\item [-] Hidden layers: $$ a_1^{(\ell)} = 1,\quad a_k^{(\ell)} = g^{(\ell)}\left( \sum_{j=1}^{n^{(\ell -1)}+1} \con_{j,k}^{(\ell-1)} a_j^{(\ell-1)} \right) $$
	for $\ell \in \{ 2,\dots, L-1 \}$ and $k \in \{2,\dots,n^{(\ell)} +1 \}. $
	\item [-] Output layer: $$a_k^{(L)} = g^{(L)} \left(\sum_{j=1}^{n^{(L -1)}+1} \con_{j,k}^{(L-1)} a_j^{(L-1)}\right) $$
	for $k \in \{1,\dots, n^{(L)}\}$
\end{itemize}
\end{definition}

Note that the output layer has no bias unit. The entry $\con_{j,k}^{\ell}$ of the weight matrix $\con^{(\ell)} \in \R^{n^{(\ell-1)} \times n^{(\ell)}}$ describes the weight from neuron $a_{j}^{(\ell-1)}$ to the neuron $a_k^{(\ell)}$. For notational convenience, we assemble all entries $\con_{j,k}^{(\ell)}$ in a vector $\R^K$ with $$K := n^{(1)}\cdot n^{(2)} + n^{(2)} \cdot n^{(3)} + \dots + n^{(L-1)} \cdot n^{(L)}.$$ For the numerical experiment we use $g^{(\ell)}=\log(1+e^x)$ for $\ell =2,\dots,N-1$ and $g^{(L)}(x)=x.$
An illustration of an NN with $L=3,$ four inputs and six units in the hidden layer can be found in Figure~\ref{fig:NN}.

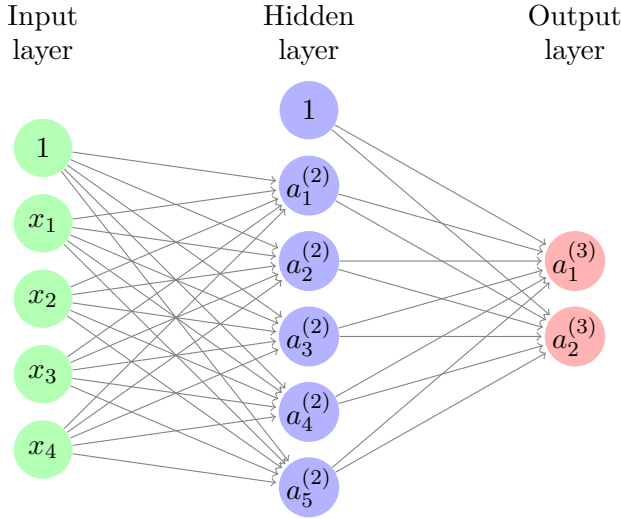
\begin{figure}[ht]
\def\layersep{3.5cm}
\begin{tikzpicture}[shorten >=1pt,->,draw=black!50, node distance=\layersep]
	\tikzstyle{every pin edge}=[<-,shorten <=1pt]
	\tikzstyle{neuron}=[circle,fill=black!25,minimum size=22pt,inner sep=0pt]
	\tikzstyle{bias neuron1}=[neuron, fill=green!30];
	\tikzstyle{input neuron}=[neuron, fill=green!30];
	\tikzstyle{output neuron}=[neuron, fill=red!30];
	\tikzstyle{hidden neuron}=[neuron, fill=blue!30];
	\tikzstyle{bias neuron2}=[neuron, fill=blue!30];
	\tikzstyle{annot} = [text width=4em, text centered]
	
	% Draw bias
	\node[bias neuron1](BN1)  at (0,-1) {};
	\node[annot] at (0,-1) {1};
	% Draw the input layer nodes
	\foreach \name / \y in {1,...,4}
	% This is the same as writing \foreach \name / \y in {1/1,2/2,3/3,4/4}
	\node[input neuron] (I-\name) at (0,-1-\y) {$x_\y$};

	\node[bias neuron2](BN2)  at (\layersep,-0.5) {};
	\node[annot] at (\layersep,-0.5) {1};
	
	% Draw the hidden layer nodes
	\foreach \name / \y in {1,...,5}
	\path[yshift=0.5cm]
	node[hidden neuron] (H-\name) at (\layersep,-1-\y) {$a_\y^{(2)}$};
	
	% Draw the output layer node
	\foreach \name / \y in {1,...,2}
	\node[output neuron] (O-\name) at (2*\layersep,-1.5-\y) {$a_\y^{(3)}$};
	
	% Connect every node in the input layer with every node in the
	% hidden layer.
	\foreach \source in {1,...,4}
	\foreach \dest in {1,...,5}
	\path (I-\source) edge (H-\dest);
	
	% Connect bias 1 with every node in hidden layer
	\foreach  \dest in {1,...,5}
	\path (BN1) edge (H-\dest);
	
	% Connect every node in the hidden layer with the output layer
	\foreach \source in {1,...,5}
	\foreach \dest in {1,...,2}
	\path (H-\source) edge (O-\dest);
	
	% Connect bias 2 with every node in output layer
	\foreach  \dest in {1,...,2}
	\path (BN2) edge (O-\dest);
	
	% Annotate the layers
	\node[annot,above of=BN2, node distance=1cm] (hl) {Hidden layer};
	\node[annot,left of=hl] {Input layer};
	\node[annot,right of=hl] {Output layer};
\end{tikzpicture}
\caption{Illustration of a feed-forward artificial network with 4 inputs and one bias input, one hidden layer with one bias unit and 5 neurons and two outputs. This corresponds to $n^{(1)} = 4, n^{(2)} = 5, n^{(3)} = 2, L=3$.}
\label{fig:NN}
\end{figure}	

%%%%%%%%%%%%%%%%%%%%%%%%%%%%%%%%%%%%%%%%%%%%%%%%%%%%%%%%%%%%
\section{Analysis of the Optimal Control problem}\label{sec:anaopt}
%%%%%%%%%%%%%%%%%%%%%%%%%%%%%%%%%%%%%%%%%%%%%%%%%%%%%%%%%%%%

In this section we analyze the parameter estimation problem in terms of well-posedness. We begin with the state system and discuss then the well-posedness of the parameter identification problem. Finally, we show the existence of an adjoint state.

\subsection{Well-posedness of the ODE system}
Under some assumption on the activation functions $g^{(\ell)}, \ell = {2,\dots,L}$ we can establish a well-posedness result for system \eqref{ODE}. As the output of the neural network is the right-hand side of the ODE, we assume $n^{(L)} = d$ for the following considerations.

\begin{theorem}[Well-posedness of the state equation]\label{thm:wellODE}
Let the activation functions, $g^{(\ell)}$ of the \emph{NN} defined in Definition~\ref{def:NN} with $ n^{(L)} = d$ be globally Lipschitz for $\ell = {2,\dots,L}.$ Then there exists a unique solution $x \in \C^1([0,T],\R^d)$ to \eqref{ODE}. 
\end{theorem}	
\begin{proof}
	The proof exploits the recursive structure of the neural network defined in Definition~\ref{def:NN}. In fact, let $x  \in \R^{n^{(1)}}$ and $a^{(L)} \in \R^d$ then 
	\begin{align*}
	a_k^{(L)}(x) &= g^{(L)}\left( \sum_{j=1}^{n^{(L -1)}+1} \con_{j,k}^{(L-1)} a_j^{(L -1)}(x) \right), \\
	a_j^{(L -1)}(x) &= g^{(L-1)}\left( \sum_{m=1}^{n^{(L -2)}+1} \con_{m,j}^{(L-2)} a_m^{(L-2)}(x) \right), \\
	&\dots,\\
	a_m^{(1)} &= x_{m-1}.
	\end{align*}	
Note that all relations are linear except for the activation functions $g^{(\ell)}, \ell = {2,\dots,L}.$
Using the globally Lipschitz assumption on $g$ and the aforementioned linearity, we obtain
\begin{equation*}
	| a_k^{(L)}(x) -  a_k^{(L)}(y) | \le C_g |x-y|,
\end{equation*}
where the constant $C_g$ depends on all the Lipschitz constants of $g^{(\ell)}, \ell = {2,\dots,L}.$ This allows us to conclude the global Lipschitz property of the right-hand side of the ODE. An application of the Picard-Lindel\"of theorem yields the well-posedness of the ODE as desired.
\end{proof}	

\begin{remark}
The SmoothReLU activation function given by $g(x) = \ln(1 + e^x)$ is one example satisfying the assumptions of Theorem~\ref{thm:wellODE}. Another suitable activation function is the identity $g(x)=x.$
\end{remark}	

The analysis of the optimal control problem is established in a Hilbert space framework. That is why we use the embedding $H^1([0,T], \R^{dN}) \hookrightarrow \C([0,T], \R^{dN})$ and define the control to state map
$$\mathcal S \colon \U \rightarrow H^1([0,T], \R^{dN}),\qquad \mathcal S \colon \con \mapsto x(\con),$$
which is well-defined thanks to Theorem~\ref{thm:wellODE}. 
Moreover, we define the reduced cost functional $$ \hat \J(\con) := \J(\mathcal S(\con); z).$$

The next theorem provides us the continuous dependence on the data for the ODE solution. It will help us in the proof of the existence of a minimizer below (Theorem~\ref{thm:exMin}).
\begin{theorem}[Continuous dependence on the data]\label{thm:contdep}
Let the assumptions of Theorem~\ref{thm:wellODE} hold and, additionally, let $\U$ be bounded and $g^{(\ell)}$ be  bounded for $\ell = {2,\dots,L-1}$. Then the solution to \eqref{ODE} depends continuously on the data, i.e.,
\[
\| x - \bar x\|_{H^1([0,T], \R^d)} \le M \left( \| x_0 - \bar x_0 \| + \| \con- \bar\con\|_{\U} \right).
\]
\end{theorem}
\begin{proof}
Let us consider the $i$-th component of the difference of two solutions corresponding to different data $\con$ and $\bar \con$ given by
\[
x_i(t) - \bar x_i(t) = (x_0 - \bar x_0)_i + \int_0^t \sum_{j=1}^N W^{i,j}_\con(x_j(s) -x_i(s)) - \sum_{j=1}^N W^{i,j}_{\bar \con}( \bar x_j(s)-\bar x_i(s))ds.
\]
We therefore need to estimate
\begin{align*}
	&|\left( W^{i,j}_{ \con}( x_j- x_i) - W^{i,j}_{\bar \con}( \bar x_j-\bar x_i) \right) _m| \\ &=  | g^{(L)}\left(\sum_{k=1}^{n^{(L-1)}+1} \con_{k,m} a_k^{(L-1)}(x_j - x_i)\right) - g^{(L)}\left(\sum_{k=1}^{n^{(L-1)}+1}  \bar \con_{k,m} a_k^{(L-1)}(\bar x_j - \bar x_i)\right) | \\
	&\le L_{g^{(L)}} \sum_{k=1}^{n^{(L-1)}+1}  |\con_{k,m} | |a_k^{(L-1)}(x_j - x_i) -a_k^{(L-1)}(\bar x_j - \bar x_i) | + |\con_{k,m} - \bar \con_{k,m} | |a_k^{(L-1)}(\bar x_j - \bar x_i) | \\
	&\le L_{g^{(L)}} \sum_{k=1}^{n^{(L-1)}+1} M_\con |a_k^{(L-1)}(x_j - x_i) -a_k^{(L-1)}(\bar x_j - \bar x_i)| + M_{g^{(L-1)}}  |\con_{k,m} - \bar \con_{k,m} |,
\end{align*}  
where $M_\con$ denotes the upper bound of $\U,$ $M_{g^{(\ell-1)}}$ denotes the bound for the activation functions $g^{(\ell)}$ with $\ell = 2, \dots,L-1$ and $L_{g^{(\ell)}}$ denotes the global Lipschitz constants of the activation functions.

The last step of the previous estimate can be recursively applied to
\begin{align*}
	&|a_k^{(\ell-1)}(x_j - x_i) -a_k^{(\ell-1)}(\bar x_j - \bar x_i)| \\ &\qquad \qquad \le \sum_{p=1}^{n^{(\ell-2)}+1} \con_{p,k}^{(\ell-2)} a_p^{(\ell-2)}(x_j -x_i) - \bar \con_{p,k}^{(\ell-2)} a_p^{(\ell-2)}(\bar x_j - \bar x_i),\quad \ell = 3,\dots,L-1. 
\end{align*}
For the first layer, we have
\[
|a_k^{(1)}(x_j - x_i) -a_k^{(1)}(\bar x_j - \bar x_i)| \le |(x_j - \bar x_j)_k| +  |(x_i - \bar x_i)_k|.
\]
The two estimates together yield
\begin{equation}\label{eq:stabilityW}
	|\left( W^{i,j}_{ \con}( x_j-x_i) - W^{i,j}_{\bar \con}( \bar x_j-\bar x_i) \right) _m|  \le C_1 \|\con - \bar \con\| + C_2 \| x - \bar x \|.
\end{equation}
This implies
\begin{align*}
	\| x(t) - \bar x(t)\| = \sum_{i=1} |x_i(t) - \bar x_i(t)| \le \|x_0 - \bar x_0\| + C_3  \|\con - \bar \con\| + C_4 \int_0^t \|x(s) -\bar x(s)\| ds.
\end{align*}
An application of Gronwall's theorem gives us 
\begin{equation}\label{eq:stability}
\| x(t) - \bar x(t) \| \le C \left( \| x_0 -\bar x_0 \| + \| \con -\bar \con \| \right) e^{C_4t}.
\end{equation}
Using \eqref{eq:stabilityW} and \eqref{eq:stability}, we obtain for the time derivative
\begin{align*}
\| \frac{d}{dt} x_i(t) - \frac{d}{dt} \bar x_i(t) \| &\le  \sum_{j} \| W^{i,j}_\con(x_j -x_i) - W^{i,j}_{\bar \con}(\bar x_j - \bar x_i) \| \\
&\le C_5\sum_{j} \left( \| x_0 -\bar x_0 \| + \| \con -\bar \con \| \right) e^{C_4t} .
\end{align*}
Summing over $i=1,\dots,N$ we get the bound
\begin{equation}\label{eq:stabilityDerivative}
	\| \frac{d}{dt} x(t) - \frac{d}{dt} \bar x(t) \| \le C \left( \| x_0 -\bar x_0 \| + \| \con -\bar \con \| \right) e^{C_4t}.
\end{equation}
Combining the estimates in \eqref{eq:stability} and \eqref{eq:stabilityDerivative} leads to the desired result.
\end{proof}

In the following we are concerned with the existence of a minimizer to the control problem and the existence of an adjoint state. Latter will help us to state the first order optimality conditions and to compute the gradient for the descent algorithm. For notational convenience, we define the operator $$e(x,\con) \colon H^1([0,T],\R^{dN} ) \times \U \rightarrow Z, \qquad e(x,\con)  = \Big( e_i(x,\con) \Big)_{i=1,\dots,N}$$ with $$e_i(x,\con) = \frac{d}{dt} x_i - \sum_{j=1}^N W_\con^{i,j}(x_j -x_i).$$

\begin{theorem}[Existence of a minimizer]\label{thm:exMin}
Let $W_\con^{i,j}$ be weakly continuous for all $i,j=1,\dots,N$ and $g^{(L)}(0)=0$. Then there exists a minimizer $\con^*\in\Uad$ for Problem~\ref{optProb}. 
\end{theorem}
\begin{proof}
As the cost is bounded from below and the state problem is well-posed by Theorem~\ref{thm:wellODE}, there exists
$$ m = \inf\limits_{\con \in \Uad} \J(\con).$$
 Moreover, we find a minimizing sequence $(\con_n)_n \subset \Uad.$ 
By definition of $\Uad$, $(\con_n)_n$ is bounded. Note that $\Uad$ is finite dimensional and thus reflexive. Hence we can extract a converging subsequence (not relabeled)  $(\con_n)_n$ with $\con_n \rightarrow \con_\infty$ as $n \rightarrow \infty.$  

The assumption $g^{(L)}(0)=0$ together with Theorem~\ref{thm:contdep} assures the boundedness of the solution operator $\mathcal S$. Hence, the sequence of solutions $(S\con_n)_n = (x_n)_n$ is bounded in the reflexive space $H^1([0,T],\R^{dN}).$ This allows us to extract a weakly convergent subsequence  $x_{n_k} \rightharpoonup x_\infty$ in $H^1([0,T],\R^{dN}).$ Together with the weak continuity of $W_\con^{i,j}$ we obtain
$$ e(x_{n_k},\con_{n_k}) \rightharpoonup e(x_\infty, \con_\infty),$$
which implies
$$ \| e(x_\infty, \con_\infty) \| \le \liminf\limits_{k \rightarrow \infty} \| e(x_{n_k},\con_{n_k}) \| = 0. $$
Therefore it holds $\mathcal S(\con_\infty) = x_\infty.$ As the cost functional is weakly lower semicontinuous w.r.t.~$x,$ we obtain
\[
\J(\con_\infty) = \J(x_\infty;z) \le \liminf\limits_{n \rightarrow \infty } \J(x_n;z) = \liminf\limits_{n \rightarrow \infty } \J(\con_n) =  m = \inf\limits_{\con \in U_\text{ad}} \J(\con).
\]
We conclude that $\con_\infty$ is a minimizer of the optimal control problem.
\end{proof}

\begin{remark}
	Note that the minimizer of the parameter identification problem may not be unique due to the nonlinearity in the state problem.
\end{remark}

\begin{remark}
	Again, the SmoothReLU activation functions satisfy the requirements of the existence result in Theorem~\ref{thm:exMin}. Indeed, $g(x) = \ln(1 + e^x)$ and $g(x)=x$ have uniformly bounded derivatives. Hence, we apply the mean-value theorem to obtain
	$$ \int_0^T (g((x_{n_j} -x_{n_i})_k) - g((x_j - x_i)_k)) \phi(t) dt = \int_0^T g'(\xi) (x_{n_j} - x_j + x_{n_i} -x_i)_k \phi(t) dt \longrightarrow 0$$
	as $x_n$ converges weakly to $x$ for some $\xi \in [(x_{n_j} - x_{n_i})_k,(x_j - x_i)_k]$.  
	Moreover, choosing $g^{(L)}$ to be the identity leads to $g^{(L)}(0)=0$. 
\end{remark}

\begin{theorem}[Existence of adjoint state]
Let $g^{(\ell)} \in \C^1$ with $(g^{(\ell)})' \in \text{Lip}_\text{loc}(\R)$ globally bounded and let $(\bar x, \bar u)$ an optimal solution of Problem~\ref{optProb}. Then there exists an adjoint state, $\bar p,$ such that the following optimality condition holds:
\begin{align*}
		&\frac{d}{dt} \bar x_i = \sum_{j=1}^N W^{i,j}_{\bar\con}(\bar x_j -\bar x_i), \quad \bar x_i(0) = z_0^i,\quad i=1,\dots,N,\\
		&\int_0^T \left( \frac{d}{dt} h_i + \sum_{i=1}^N d_{x_i} W_{\bar\con}^{i,j}(\bar x_j - \bar x_i) h_i - d_{x_j} W_{\bar\con}^{i,j}(\bar x_j - \bar x_i) h_j \right) \cdot \bar p_i\, dt + h_i(0) \cdot \eta \\ &\qquad\qquad\qquad\qquad = \int_0^T (\bar x_i(\bar \con)-z_i) \cdot h_i\,dt, \quad\qquad\forall\, h_i \in H^1((0,T), \R^d),\, \eta \in \R^d, \\
		& \int_0^T \sum_{i=1}^N \sum_{j=1}^N \nabla_\con W_{\bar \con}^{i,j}(\bar x_j -\bar x_i) \cdot \bar p_i \,dt \cdot (\bar \con - \con ) \ge 0 \quad \text{for all}\quad \con \in U_\text{ad}.
\end{align*}
\end{theorem}
\begin{proof}
We aim to apply Corollary 1.3 in \cite{HUUP} and check its four requirements. 

We begin the set of admissible controls. $\Uad = [-1,1]^K \subset U = \R^K$ is nonempty, convex and closed. This verifies the first requirement.

Next, we have to show that $\J \colon Y \times \U \rightarrow \R$ and $e \colon Y \times \U \rightarrow Z$ are continuous Fr\'echet differentiable. We set $\U = \R^K, Y = H^1((0,T), \R^d)$ and $Z = L^2((0,T), \R^d)$. Note that these are all Banach spaces. The cost functional $J$ is Fr\'echet differentiable by standard arguments, see for example \cite{HUUP}, it holds
$$ d_y J(y(u),u)[h] = \int_0^T (x(u) - z_i) \cdot h\, dt, \qquad d_u J(y(u),u)[h] = 0.$$
 For the state operator $e$ we find
\begin{align*} d_{y} e_i(y,u)[h] &= \frac{d}{dt} h_i + \sum_{j=1}^N d_{x_i} W_\con^{i,j}(x_j - x_i) h_i - d_{x_j}W_\con^{i,j}(x_j - x_i) h_j,\quad i=1,\dots,N, \\
d_{u} e_i(y,u)[k] &= \int_0^T \sum_{i=1}^N \sum_{j=1}^N d_\con W_{ \con}^{i,j}( x_j - x_i)[k] \cdot \bar p_i \,dt.	
\end{align*}
Using the assumptions on the activation function $g^{(\ell)}$, this yields $e_y(y,u) \in \mathcal L(Y,Z).$ 

Third, Theorem~\ref{thm:wellODE} assures that $e(y,u)=0$ admits a unique solution $$y=y(u) \in \C^1((0,T),\R^d) \subset H^1((0,T), \R^d).$$

We are left to show that $e_y(y(u),u) \in \mathcal L(Y,Z)$ has a bounded inverse for all $u \in V$ in a neighborhood of $\Uad.$ In order to see that we consider
$$ e_y(y(u),u) [h] = r $$
for arbitrary $r \in Z=L^2((0,T), \R^d).$ By the Caratheodory theory for ODEs, there exists a unique solution to this equation for every $r \in  Z.$ Using the explicit expression of $e_y(y(u),u)$ we find
$$ \| h(t) \| = \| h(0) \| + \int_0^T \| r(s) \| ds + C_{DW}\int_0^T \| h(s) \| \quad \text{for all }t\in(0,T) $$
for some constant $C_{DW}>0$ depending on the global bounds on the derivatives of the activation functions. An application of Gronwall's Lemma yields the boundedness of the inverse of $e_y(y(u),u).$

Hence, the requirements of Corollary 1.3 in \cite{HUUP} are satisfied and we find an adjoint state $\bar p$ such that the optimality condition holds as desired.
\end{proof}

These results assure the well-posedness of the optimization problem and its first order optimality system. We are well-equipped to state the stochastic gradient descent algorithm that we propose for the treatment of general neural network driven optimization problems. The numerical results below are computed with this algorithm as well.

%%%%%%%%%%%%%%%%%%%%%%%%%%%%%%%%%%%%%%%%%%%%%%%%%%%%%%%%%%
\section{Stochastic Gradient descent algorithm}\label{stochgrad}
%%%%%%%%%%%%%%%%%%%%%%%%%%%%%%%%%%%%%%%%%%%%%%%%%%%%%%%%%%%

Having the results of the previous section at hand, we can now propose the stochastic gradient descent algorithm. 
In this section we  assume that $p \in H^1((0,T),\R^d) \subset Z$. This allows us to formulate the adjoint system in strong form given by
\[
-\frac{d}{dt} \bar p_i = \sum_{j=1}^N \nabla_{x_i} W_{\bar \con}^{i,j}(\bar x_j - \bar x_i) (\bar p_i - \bar p_j) - (\bar x_i -z_i),\qquad i=1,\dots,N,
\]
supplemented with the terminal condition $p(T) = 0,$ where $\bar x$ is a solution to \eqref{ODE} with initial condition $x_0 = z_0.$ The gradient is based on this strong form of the adjoint.

\subsection{Gradient of the reduced cost functional}

We compute the gradient of the reduced cost functional as follows
\[
\langle \hat \J'(u), s \rangle = \langle z - \mathcal S(u) , \mathcal S'(u)[s] \rangle = \langle \mathcal S'(u)^* (z - \mathcal S(u)) , s \rangle .
\]
Let $e$ be the operator defined above, we obtain
\[
S'(u)^* (z - \mathcal S(u)) = - e_u(\mathcal S(u),u)^* e_y(\mathcal S(u),u)^{-*}(z - \mathcal S(u)) = e_u(\mathcal S(u),u)^* p,
\]
where we used the adjoint equation $e_y(\mathcal S(u),u)^* p = - (z-\mathcal S(u)).$ Altogether, the gradient of the reduced cost functional can be expressed as
\[
\hat \J'(u) = e_u(\mathcal S(u),u)^* p(u) = \int_0^T \sum_{i=1}^N \sum_{j=1}^N \nabla_\con W_{\bar \con}^{i,j}(\bar x_j -\bar x_i) \cdot \bar p_i \,dt.
\]

Based on these considerations we may establish a gradient descent algorithm.  Moreover, for $\Uad$ bounded, we may employ a projected gradient descent method. Note that in our application we face big data sets and therefore the evaluation of the full cost function is very costly. This is why we use a mini-batch algorithm. Its details are discussed in the following section.

\begin{remark}
For general artificial neural networks the computation of $\nabla_x W_\con^{i,j}$ and $\nabla_\con W_\con^{i,j}$ can be very complicated. Here, the recursive structure of the artificial neural network allows for a straight forward computation of the derivatives using multiple application of the chain rule. 
\end{remark}

\subsection{Stochastic descent}\label{sec:stoDescent}
In many applications we expect the cost functions to have several local minima, where usual gradient descent algorithms may get stuck. To prevent this issue when training the neural networks, we use a mini-batch gradient descent scheme. Indeed, we compute ADADELTA updates as proposed in \cite{GoettKnapp}. This leads to the following algorithm:

Let $\alpha_k$ denote the $k$-th iterate of the gradient descent. We define
\begin{gather*}
	\Delta\alpha_k = \alpha_{k+1} -\alpha_k,\qquad E[g^2]_0 = 0, \qquad E[\Delta \alpha^2]_0 = 0, \\
	E[g^2]_k = \rho E[g^2]_{k-1} + (1-\rho) (\nabla \J_{k-1}^{\tilde m}(\alpha_{k-1}))^2, \\
	\Delta \alpha_k = -\frac{\sqrt{E[\Delta \alpha^2]_{k-1} +\epsilon}}{\sqrt{E[g^2]_k} + \epsilon}\nabla \J_k^{\tilde m}(\alpha_k), \\
E[\Delta \alpha^2]_k = \rho E[\Delta \alpha^2]_{l-1} + (1-\rho) \Delta\alpha_k^2,
\end{gather*}
where $\rho \in (0,1)$ is the rate for the adaption of the squared gradient information and $\epsilon >0$ avoids singular values by division. Both are fixed parameters. 

Note that these updates are still deterministic. In order to incorporate noise that help us to escape from local minima, we add a multivariate normal distributed random vector  $N_k$ with $N_k \sim \mathcal N(0,\Sigma_k)$ to the gradient in each iteration. Here, $\Sigma_k$ denotes the variance matrix. We choose $$ (\Sigma_k)_{ii} = \frac{\eta_1}{(1+k)^{\eta_2}} $$ for some constants $\eta_1, \eta_2 >0$ and  $ (\Sigma_k)_{ij} = 0 $ whenever $i \ne j.$ Note that the noise diminishes as the number of iteration, $k$, increases. For the numerical simulations we set $\eta_1 = 1$ and $\eta_2 = 0.55,$ the adaption rate $\rho = 0.95$ and $\epsilon = 10^{-6}.$

%%%%%%%%%%%%%%%%%%%%%%%%%%%%%%%%%%%%%%%%%%%%%%%%%%%%%%%%%%%%
\section{Parameter estimation based on traffic data}\label{traffic}
%%%%%%%%%%%%%%%%%%%%%%%%%%%%%%%%%%%%%%%%%%%%%%%%%%%%%%%%%%%%
The stochastic gradient descent method proposed above can be used to treat general neural network driven optimization or parameter identification problem. In the following we apply it to two applications for which we have real data available. We begin with a one-dimensional traffic dynamic modelled with the help of a first order ODE system, then we consider a two-dimensional pedestrian dynamic modelled with a second order dynamic.
Moreover, we use the stochastic gradient descent algorithm to estimate parameters of well-known interaction forces for the two scenarios and compare the output and the cost.

\subsection{Traffic models}
For the traffic dynamic we assume to have a \textit{follower-leader} dynamic. Therefore, we choose two microscopic versions of the well-known LWR-model \cite{LWR_HR} with logarithmic and linear velocity function, respectively, as reference models.

\subsubsection{Traffic dynamic with LWR-Model}
Let $x_i(t)$ denote the position of the $i$-th car at time $t \in [0,T].$ Then the evolution of the cars is given by
\begin{subequations}\label{eq:LWR}
\begin{align}
\frac{d}{dt} x_i(t) &= W_u^{i,i+1}\left(\frac{x_{i+1}(t) -  x_i(t)}{L} \right), \quad i=1,\dots, N-1, \\
\frac{d}{dt} x_N(t) &= v_0. 
\end{align} 
\end{subequations}
The parameters are $v_0$ the velocity of the leading car and $L$ the length of the cars.
We consider $W_u^{i,i+1}(z) = v_0 \log(z)$ and $ W_u^{i,i+1}(z) = v_0 (1 - 1/z) $for $z > 0.$ The task of the parameter identification is to estimate $\con = (L, v_0) \in \mathbb R^2.$

\subsubsection{Traffic dynamic with artificial neural network}
The model driven by the feed-forward neural network is given as follows. We parametrize the interactions of the follower-leader dynamic by $W_\con,$ where $\con = (v_0, \con_\text{Net})$ and $\con_\text{Net}$ is assumed to contain all the information of the neural network. The dynamic is then given by
\begin{subequations}\label{eq:NNvelo}
	\begin{align}
		&\frac{d}{dt} x_i(t) = W^{i,i+1}_\con(x_{i+1}(t) -  x_j(t)), \quad i=1,\dots, N-1, \\
		&\frac{d}{dt} x_N(t) = v_0. 
	\end{align} 
\end{subequations}
supplemented with initial data $x(0) = x_0.$ 
\begin{remark}
	Note that we have to prescribe some value for the first vehicle even in this case, as we assume that the behaviour of cars depends on the distance to the vehicle in front and the first vehicle has no one in front.
\end{remark}

\subsection{Data processing of traffic data set}
We use the microscopic traffic data set that was recorded within the project ESIMAS \cite{dataTraffic}. It contains vehicle data from 5 cameras that were placed in a $1 km$ tunnel section on the German motorway A3 nearby Frankfurt / Main. For the parameter estimation we extract sequences from the data that contain three or more vehicles in one lane. For simplicity, we restrict our considerations to the middle lane data and neglect the data of the $y-$coordinate. Thus, we have one-dimensional traffic data for the parameter estimation.

First, we interpolate the position data that is supplemented with time stamps to a reference time discretization. Having all the data aligned to the reference time discretization, we filter sequences of data where two or more vehicles are present in the camera frame. This yields a database with various sequences of different length and with different number of vehicles that we use for the parameter identification. 

Note that after this data extraction, the vectors containing the positions of the cars for each sequence are ordered. In fact, we pass this ordered vector to the neural network and hold on to the assumption that the interaction of the vehicles depends on the distance to the vehicle in front. This is why we have to prescribe some velocity for the first car in the data set, even in the case of the neural network without physical parameters, see \eqref{eq:NNvelo}.

\begin{remark}
	Note that even though the parameter set gives data on the vehicle type, we do not use this information in the approach. Therefore, we expect to obtain an averaged length $L$ as output of the parameter estimation for the LWR-model with logarithmic and linear velocity function.
\end{remark}

\subsection{Numerical schemes}
For both approaches, the LWR and the NN model, we solve all the parameter identification problems with the stochastic gradient descent method discussed in Section~\ref{sec:stoDescent}. The forward and adjoint models are integrated using an Explicit Euler scheme. The traffic data set has time step $dt_\text{data}=0.2$ for the simulation we use a finer discretization, i.e., $dt = 0.002$. 

\subsection{Numerical results}
For the numerical results we use the two data sets of "day 1" of all five cameras. We test three different neural network settings with two, four and ten neurons in the hidden layer and call them, N2, N4, N10, respectively. The results corresponding to the LWR-model with logarithmic velocity function are denoted by "Log" and the ones obtained with the linear velocity function are denoted by "Lin".

Note that the provided data units are meter $m$ and seconds $s$. We initialize the velocity $v_0$ and car length $L$ for the LWR cases with $v_0 = 30 \frac{m}{s}$ and $L = 5 m.$ The same initial velocity is used for all neural network dynamics. The weights for the neural networks are initialized with random values uniformly distributed in $[-1,1]$. We set a lower bound for the car length $L_{\min} = 2 m.$
\begin{figure}[ht!]
	\includegraphics[scale=0.3]{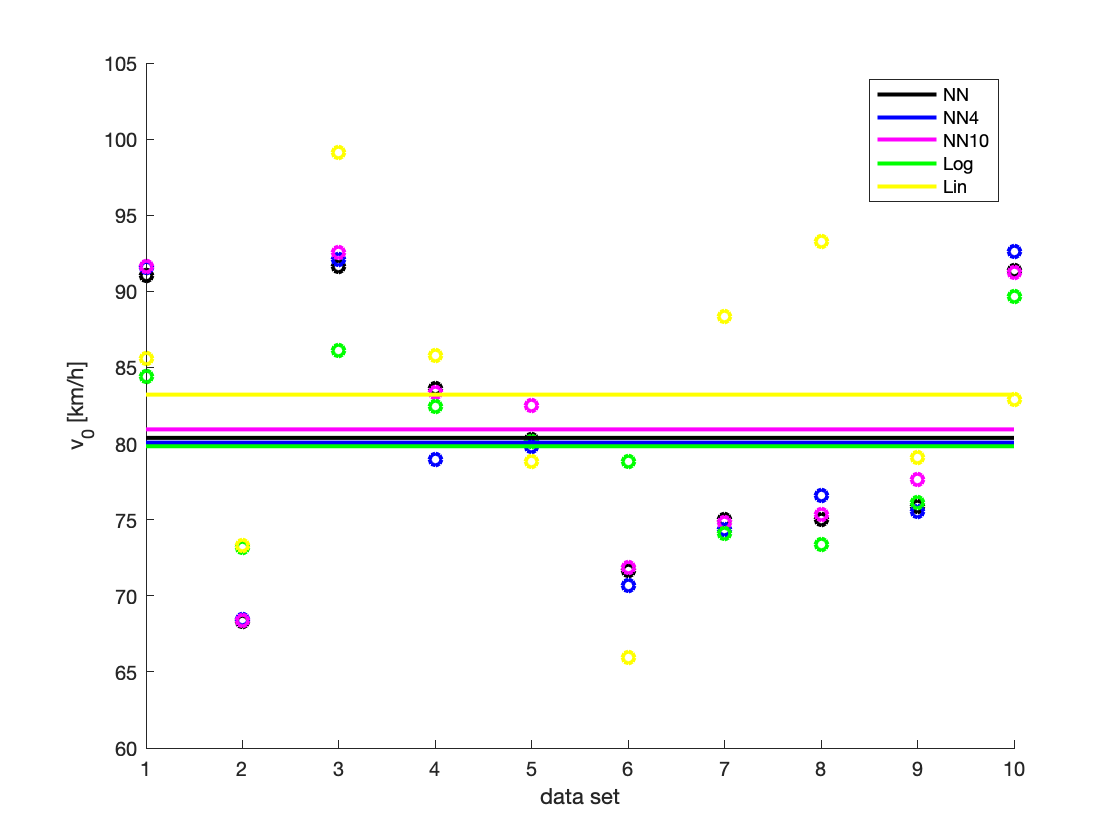}	
	\caption{The circles show the identified velocities of the leading car for the different methods (colour coded) and the different data sets. The lines show the mean velocities taken over all data sets for the different methods.}
	\label{fig:veloLength}
\end{figure}
Figure~\ref{fig:veloLength} shows the results of the identification for the different models. The circles show the identified velocities of the leading car for the different methods and the different data sets. The lines show the mean velocities taken over all data sets for the different methods. The mean velocity of the LWR-model with linear velocity function is slightly larger than the other mean velocities which are close to $80 \frac{km}{h}$. The maximal speed allowed on the highway is $100\frac{km}{h}.$ We see differences between the models for single data sets. The car lengths identified for the two LWR approaches are shown in Table~\ref{tab:length}.
\begin{table}[ht!]
	\begin{tabular}{l | c | c | c | c | c | c |c | c | c | c | c | c |}
		& 1 & 2 & 3 & 4 & 5 & 6 & 7 & 8 & 9 & 10 & average \\
		\hline
		Lin & 3.47 &    3.95 &    5.64 &    4.50 &    2.25 &    2.54 &    8.33 &    7.07 &   2.00 &    7.16 &   4.69  \\
		\hline
		Log & 8.24 &   7.92  &  9.20 &   9.90  &  7.17 &   6.63 &   9.98 &   9.91 &    5.52 &    9.76   & 8.43 \\
		\hline 
	\end{tabular}
	\vspace{0.5em}
	\caption{ Car lengths (in $m$) estimated with the algorithm for the 10 data sets with the LWR-model with linear and logarithmic velocity.}
	\label{tab:length}
\end{table}

 The average car length estimated with the linear velocity approach is $4.69 m$ and the one for the logarithmic velocity function is $8.43 m.$ These numbers may indicate that the linear model estimates the true car length, while the logarithmic approach includes the distance to the next car into this value. For test case 9 the car length of the linear model hits the lower bound $L_{\min}.$ Out of curiosity we dropped the lower bound assumption in this case, leading to a value of $0.6 m$ and an over all average velocity close to $80 \frac{km}{h}$ similar to the other models. 
\begin{table}[ht!]
	\begin{tabular}{l | c | c | c | c | c | c |c | c | c | c | c | c |}
		& 1 & 2 & 3 & 4 & 5 & 6 & 7 & 8 & 9 & 10 & average \\
		\hline
		N2 & 54.68 &  42.10 &  98.67 &  33.38 & 24.10 &  20.01  & 39.73 &  53.09 &  10.37 &  74.27 & 45.04\\
		\hline
		N4 & 54.25 &  \textbf{40.98}  & 108.80 &  \textbf{33.12}  & \textbf{20.68}  & \textbf{18.26} &  40.43 &  \textbf{53.05}  &  \textbf{8.11} &  72.02 & \textbf{44.97} \\
		\hline
		N10 & 50.43 &  46.02 &  \textbf{94.69} &  41.21 &  24.81 &  19.97 &  43.35 &  57.21 &   9.23  &  69.61 & 45.65 \\
		\hline
		Lin & \textbf{49.11} &   41.33 &   96.91 &   39.17 &   23.11 &   34.11 &   \textbf{39.39} &   62.30 &   15.21 &   \textbf{63.78} & 46.44 \\
		\hline
		Log & 58.88 &  50.04 & 113.69 &  71.31 &  32.40 &  59.59 &  39.68 &  57.72  & 19.84 &  89.24  & 59.24 \\
		\hline
	\end{tabular}
	\vspace{0.5em}
	\caption{Evaluations of the cost functional with the interaction forces identified by the algorithm for the 10 data sets. The least cost value for each column is highlighted.}
	\label{tab:cost}
\end{table}

Table~\ref{tab:cost} shows the cost for the different approaches using the interaction forces and parameters identified by the algorithm. In average the neural network approach with four neurons in the hidden layer performs best. In fact, all the neural network approaches perform better than the LWR-model with linear or logarithmic velocity function. The least cost for each data set and for the average are highlighted. N4 has the least cost in $60\%$ of the test cases. N10 approximates the third data set best and Lin gives the best result for three of the data sets. The linear model is closer to the NN approaches and outperforms the logarithmic model.

Figure~\ref{fig:interactionForce} illustrates the different interaction forces resulting from the parameter estimation. The interaction forces of the neural network models are rather linear except for the range $2 m - 5 m.$ The interaction forces of the linear and the logarithmic LWR-model behave different in this region, in fact, they have very steep gradients and enter a negative regime, which corresponds to slowing down. For distances in the range of $15 m - 50 m$ the linear LWR-model is close to the interaction forces of the neural networks, whereas the logarithmic LWR-model admits larger values.
\begin{figure}[ht!]
	\includegraphics[scale=0.27]{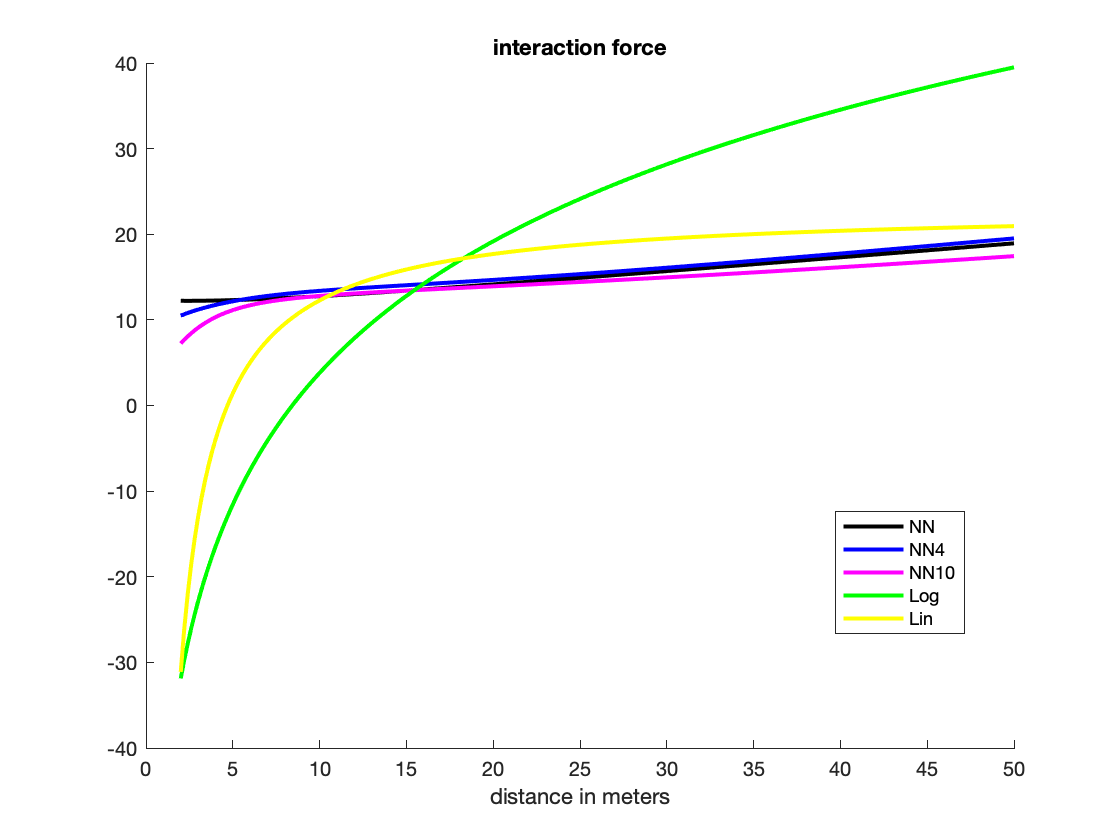}	
	\caption{Interaction forces resulting from the parameter identifications.}
	\label{fig:interactionForce}
\end{figure}

%%%%%%%%%%%%%%%%%%%%%%%%%%%%%%%%%%%%%%%%%%%%%%%%%%%%%%%%%
\section{Parameter estimation based on crowd data}\label{crowd}
%%%%%%%%%%%%%%%%%%%%%%%%%%%%%%%%%%%%%%%%%%%%%%%%%%%%%%%%%%

The following parameter estimation is based on a data set provided by the Institute for Advanced Simulation:~Civil Safety Research of Forschungszentrum Jülich \cite{PedestrianData}. In particular, the data set of a bidirectional flow in a corridor is used. We employ the stochastic gradient descent algorithm to fit the parameters of the neural  network as well as the social force Model.
\subsection{Crowd data with social force Model}
We assume that the social force model proposed by Helbing and Moln\'{a}r \cite{HelbingMolnar1998} is a good fit to work with the crowd data set. It reads
\begin{subequations}\label{eq:SFM}
\begin{align}
	\frac{d}{dt} x_i &= v_i, \\
	\frac{d}{dt} v_i &= v^D_i + \frac{1}{Nm} \sum_{j=1}^N F_{i,j} + \frac{1}{Km} \sum_{k=1}^K F_{iw}
\end{align}
\end{subequations}	
supplemented with initial data $x(0) = x_0$ and $v(0) = v_0.$ The relevant parameters are given in Table~\ref{tab:parameters}.

\begin{table}
\begin{center}
\begin{tabular}{c | c }
parameter & variable  \\ \hline
mass & $m$ \\ radius & $r$ \\
relaxation time & $\tau$ \\ force constant & A \\
force constant & $B$ \\force constant & $\kappa$ \\
force constant & $k$ \\ desired velocity of $i$-th pedestrian & $v_i^\text{des}$  \\
number of wall discretization points & $M$ \\ number of pedestrians & $N$ \\ \hline
\end{tabular}
\end{center}
\caption{Parameters used for the parameter estimation based on crowd data. For the parameter identification we fix values for $m, r, \tau,M,N,B$ and seek to find $A,\kappa$ and $k$.}
\label{tab:parameters}
\end{table}
\medskip

The relaxation term towards the desired velocity $v_i^\text{des}$ is constructed from the given trajectories as follows
\begin{equation}\label{eq:relaxation} 
v^D_i = \frac{1}{\tau} \big( v_i^\text{des}(t)  - v_i(t)\big), \quad \text{where} \quad
 v_i^\text{des}(t) = \frac{x_D - x_i(t) }{\| x_D - x_i(t) \|} \| v_i(t) \|. 
 \end{equation}
Here, we assume that each pedestrian tries to head towards his or her destination which is given by the last position of the data sequence and keeps the current speed. The other force terms are assumed to be given as
\begin{equation}\label{eq:InteractionForceHelbing} 
F_{ij} = F(2r - d_{i,j},v_i-v_j) = \left(A \exp\Big(\frac{2r - d_{ij}}{B}\Big) + k\, h(2r-d_{ij})\right) n_{ij} + \kappa\, h(2r-d_{ij}) \Delta v_{ji}^t t_{ij} 
\end{equation}
with
\begin{gather*}
d_{ij} = \|  x_i - x_j\|,\quad n_{ij} = \frac{x_i - x_j}{d_{ij}}, \quad t_{ij} = (-n_{ij}^2, \quad n_{ij}^1),\quad \Delta v_{ji}^t = (v_j - v_i) \cdot t_{ij} 
\end{gather*}
being the distance of pedestrian $i$ and pedestrian $j,$ the normalized vector pointing from pedestrian $j$ to pedestrian $i,$ the tangential direction and the tangential velocity difference, respectively, and,
$ h(y) = H(y)\, y $ with Heaviside function $H.$ The interaction with the walls is parametrized using
$$ F_{iw} =  F(r - d_{iw},v_i) = \left(A \exp\Big(\frac{r - d_{iw}}{B}\Big) + k\, h(r-d_{iw})\right) n_{iw} + \kappa\, h(r-d_{iw}) (v_i \cdot t_{iw}) t_{iw}. $$
We assume here and in the following that the walls consist of stationary points, such that $d_{iw}, n_{iw}$ and $t_{iw}$ are easy to compute. We fix the radius $r = 0.25,$ the scaling $B = 0.1,$ the relaxation parameter $\tau = 0.5$ and the mass $m=1$ for all simulations.
To summarize, the relevant parameters to find via the estimation procedure are $\con = (A,k,\kappa).$

\subsection{Crowd data with neural networks}
As we aim to compare the results of the model-based approach to the neural network approach, we pass the same data to neural network that models the acceleration. Indeed, we assume to neural network dynamic to be given by
\begin{subequations}\label{eq:NNseparate}
\begin{align}
	\frac{d}{dt} x_i &= v_i, \\
	\frac{d}{dt} v_i &= v^D_i + \frac{1}{mN} \sum_{j=1}^N W_u^{i,j}(x_i-x_j,v_i-v_j) + \frac{1}{mN_\text{wall}} \sum_{k=1}^{N_\text{wall}} W_u^{i,w}(x_i-x_k,v_i-v_k)
\end{align}	
\end{subequations}
supplemented with initial data $x(0) = x_0$ and $v(0) = v_0.$ %\simone{$K$ is already used for the dimension!}
 Here, we refer with $x_k$ to the positions of the discretization points of the wall and with $v_k$ to artificial velocity vectors of the wall points for $k=1,\dots,N_\text{wall}$. Using separate neural networks for the interaction and the walls allows us to compare the structure of the resulting terms one by one and to understand the two approaches in more detail. %\simone{better to introduce only the huge NN approach} 
 Note that we use the same neural network with different normal vectors for the different walls. As for the social force approach we fix the relaxation parameter $\tau = 0.5$. The other parameters need to be estimated. For notational convenience, we define
$$ \con = (\con_\text{NN}^\text{int},\con_\text{NN}^\text{wall}),$$
where $\con_\text{NN}^\text{int}$ denotes all the parameters involved in the neural network modelling the pairwise interaction between pedestrians and $\con_\text{NN}^\text{wall}$ the parameters of the neural network modelling the interaction of with the walls, respectively.

\subsection{Processing of the crowd data set}
The crowd data set is processed with the same approach as the traffic data sets. Indeed, we split it into sequences of fixed length $T_\text{seq}$. Then we consider only the pedestrians that are present during the whole sequence. In fact, we obtain $T/T_\text{seq}$ sequences with a different number of pedestrians present. This allows us to use a fixed number of pedestrians in the models for each computation of the gradient in the parameter estimation. We expect to have only small errors raised by the fact that some of the pedestrians are neglected.

After the data processing we have the trajectories of all pedestrians present in each sequence. We use the first and the last point of these trajectories to compute the relaxation term \eqref{eq:relaxation} in each time step. Moreover, we compute the velocities of the pedestrians using a finite difference approximation. This applies to both, the social force and the NN approach. 

\subsection{Numerical schemes and parameters}
We use the Explicit Euler scheme to solve the state system and the adjoint systems. The information is then passed to the stochastic descent algorithm, see Section~\ref{sec:stoDescent}.

Each sequence of the preprocessing involved $25$ time steps of length $dt = 0.04,$ leading to $T_\text{seq}=1.$ We use the same time step $dt = 0.04$ for the Euler method. The parameter of the stochastic descent algorithms are given in Section~\ref{sec:stoDescent}. The values for the initial neural  networks are a random sample that is chosen independently and uniformly distributed on $[-1,1]^K.$ We have four input variables, a hidden layer with four neurons and two output variables representing the interaction force in $x$- and $y$-direction. 

The initial values for the social force model are a random sample uniformly distributed in the interval $[0,50]^3.$
As the parameters of the social force model are assumed to be non-negative, we set them to zero, if they became negative in some iteration of the gradient descent algorithm.

\subsection{Numerical results}
In the following we discuss the results of the parameter identification that we computed with the help of the stochastic gradient descent methods derived in Section~\ref{sec:stoDescent}.
\subsubsection{Results for social force model}
First, we discuss the numerical results obtained for the social force model. We begin with plots similar to \cite{GoettKnapp}, where we show the results for the interaction forces by four contour plots for the $x$-direction and for the $y$-direction. For each of the plots we fix a vector $v = (v_1,v_2)$ which describes the difference of the velocity vectors of two interaction pedestrians. For example, a pedestrian with velocity vector $(1,0)$ interacting with a pedestrian with velocity vector $(-1,0)$ leads to a difference vector $v = (2,0).$ The contour colours show the value of the force components. Negative values correspond to repulsive forces and positive values to attraction force.

\begin{figure}
	\includegraphics[scale=0.3]{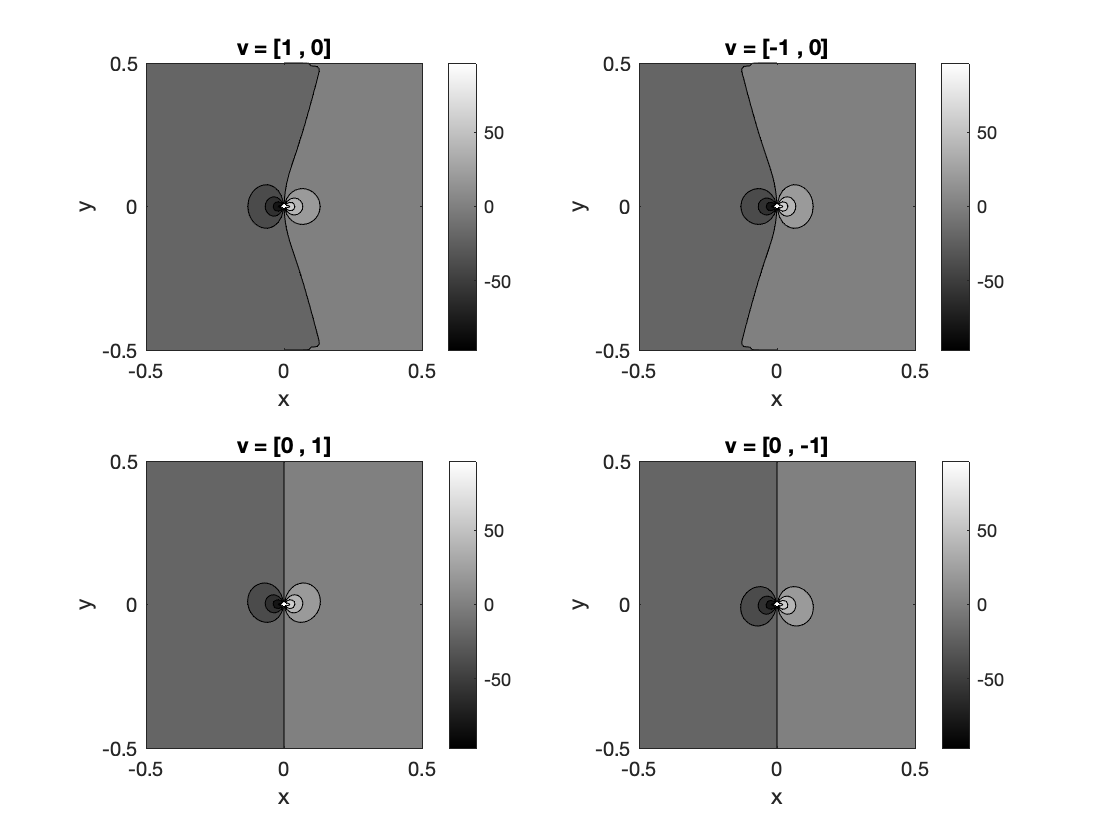}
	\caption{First component of the interaction force for social force approach with optimized parameters.}	
	\label{fig:HelbingG1}	
\end{figure}

\begin{figure}
	\includegraphics[scale=0.3]{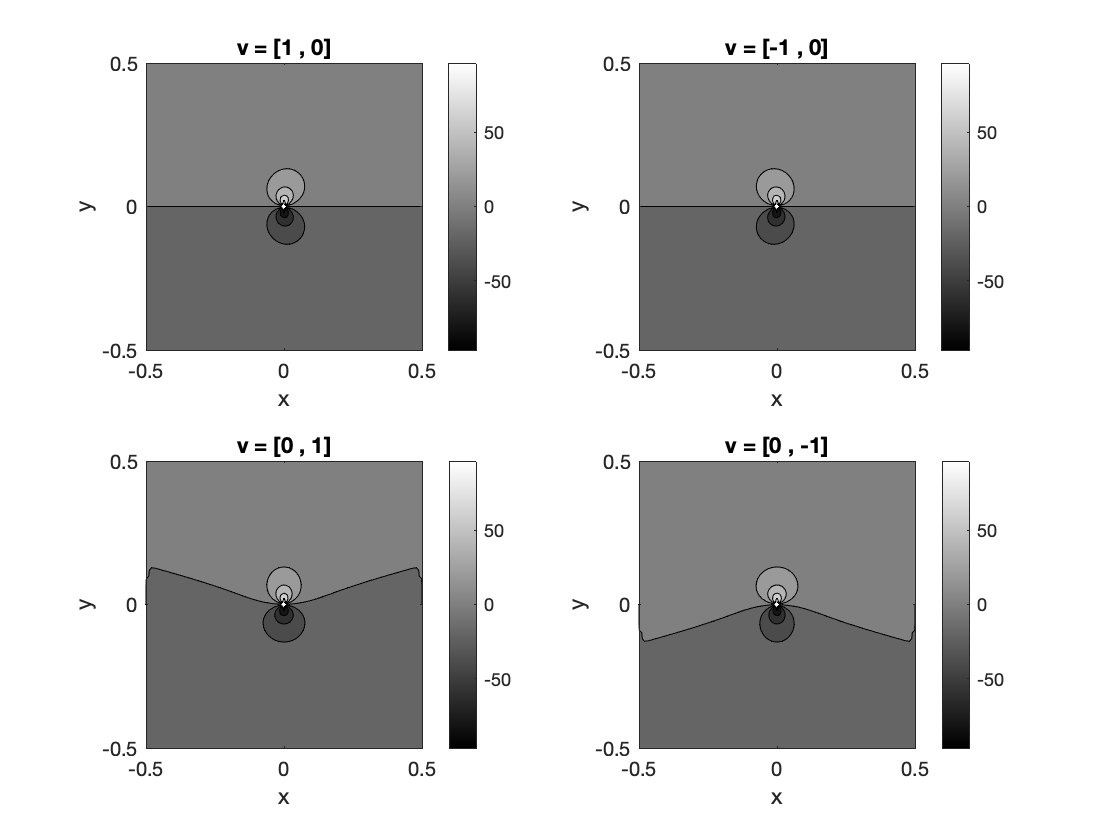}
	\caption{Second component of the interaction force for social force approach with optimized parameters.}	
	\label{fig:HelbingG2}	
\end{figure}

The optimal parameters estimated are $A = 0.0044, k=34.9539, \kappa = 9.8894.$ It is interesting to note that $A,$ the prefactor of the exponential term, is almost switched off by the optimization. In other words, only the $k$- and $\kappa$-term are active, this reduces the interaction to a very short range, as both terms are multiplied by the Heaviside function $H(2r - d),$ compare to \eqref{eq:InteractionForceHelbing}. We therefore restrict the area of the contour plot in Figure~\ref{fig:HelbingG1} and Figure~\ref{fig:HelbingG2} to $[-0.5,0.5]^2.$
In fact, we even see in Figures~\ref{fig:HelbingG1} and ~\ref{fig:HelbingG2} that strong forces occur in even shorter ranges around zero. In addition to these plots, we illustrate interaction forces for six settings in Figure~\ref{fig:6Helbing}. Every subplot show two interacting pedestrians at the positions marked with a blue and a red dot, respectively. Moreover, the velocity vectors of the pedestrians are shown in blue and red as well. The black arrow is the force vector resulting from the interaction of the two.

\begin{table}
\begin{tabular}{| c | c | c | c | c | c | c | c | }
	\hline 
	 & $x_\text{blue}$ & $x_\text{red}$ & $v_\text{blue}$ & $v_\text{red}$ & $\text{force}_\text{blue}$ & $\text{force}_\text{red}$ \T\B \\
	 \hline 
	S1 & $(0;0.22)$ & $(0;-0.22)$ & $(0;-1)$ & $(0;1)$ & $(0;2.1118)$ & $(0;-2.1118)$ \T\B\\
	S2 & $(0;0.22)$ & $(0;-0.22)$ & $(0;1)$ & $(0;1)$ & $(0;2.1118)$ & $(0;-2.1118)$ \T\B\\
	S3 & $(0.01;0.22)$ & $(-0.01;-0.22)$ & $(0;-1)$ & $(0;1)$ & $(0.0417;2.0961)$ & $(-0.0417;-2.0961)$ \T\B \\
	S4 & $(0.01;0.22)$ & $(-0.01;-0.22)$ & $(0;1)$ & $(0;1)$ & $(0.0952;2.0936)$ & $(-0.0952;-2.0936)$ \T\B\\
	S5 & $(0.22;0)$ & $(-0.22;0)$ & $(0;-1)$ & $(0;1)$ & $(2.1118;1.1867)$ & $(-2.1118;-1.1867)$ \T\B\\
	S6 & $(0.22;0)$ & $(-0.22,0)$ & $(0,1)$ & $(0,1)$ & $(2.1118;0)$ & $(-2.1118;0)$ \T\B\\
	\hline 
\end{tabular}
\vspace*{4mm}
\caption{Initial data for the study of the forces shown in Figure~\ref{fig:6Helbing}.}
\label{tab:data6settings}
\end{table}	

\begin{figure}
\includegraphics[scale=0.4]{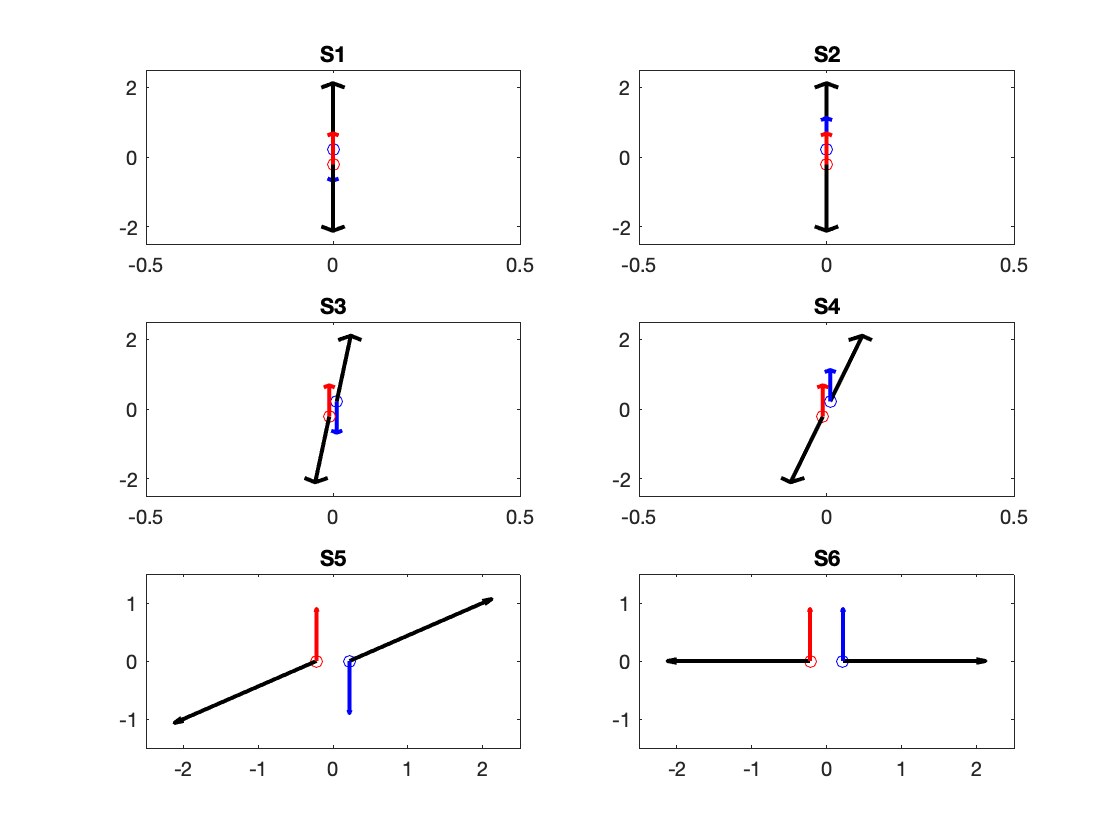}
\caption{Interaction forces resulting from social force model for different settings, see Table~\ref{tab:data6settings} for more details. The positions of the two interacting pedestrians is marked by the blue and red dot, respectively. Their velocity vectors are the blue and red arrows and the resulting forces are depicted as black arrows. On the left side the pedestrians face each other, on the right side the pedestrians walk in the same direction.}
\label{fig:6Helbing}
\end{figure}

The data of the different settings is given in Table~\ref{tab:data6settings}.
Setting S1 and S2 show a well-known problem of interaction schemes with radially symmetric forces. The interaction forces of two pedestrians encountering each other are aligned along the vector of the differences of the pedestrians positions. Therefore, there is no evasive behaviour visible in this study. Nevertheless, in combination with the relaxation term that drives the pedestrian towards their desired destination, we expect to have a reasonable model. This problem for radially symmetric interactions is well-known, and reported for example in \cite{CollisionAvoidance}. Therein, an approach to solve this issue presented as well. 

For slightly shifted positions (see S3 and S4) we see evasion behaviour of the pedestrians. The pedestrians are slowing down and slightly moving to the left or right, respectively. The force directions are similar for the case where the velocity vectors point at each other and also when the velocity vectors are aligned. In the settings S5 and S6 we see how two pedestrians walking next to each other interact. In both settings there is a strong repulsion that pushes the pedestrians away from each other. The walls have no influence on this behaviour, as their influence is only in very short range and the positions are centred in the domain.
Altogether, the resulting forces are reasonable. They model some kind of evasive behaviour, which is expected as they become active only if the two interacting pedestrians are very close together.

\subsubsection{Results for NN model}
Let us now discuss the interactions based on the neural network model. 
We consider the same plots as for the social force model above and combine the forces resulting from the neural network modelling the interaction with the forces resulting from the walls.  Figures \ref{fig:NN_G1} and \ref{fig:NN_G2} show contour plots of the forces resulting from the neural network approach when we sum up the interaction force and the wall force. Comparing these Figures to the ones resulting from the social force model, we see that the interaction strength is smaller but the interaction range is longer for the NN model. We analysed the forces resulting from the interactions and the walls separately as well, it turns out that the forces to not represent pairwise interaction between pedestrians and walls, as we intended to model.

\begin{figure}
	\includegraphics[scale=0.3]{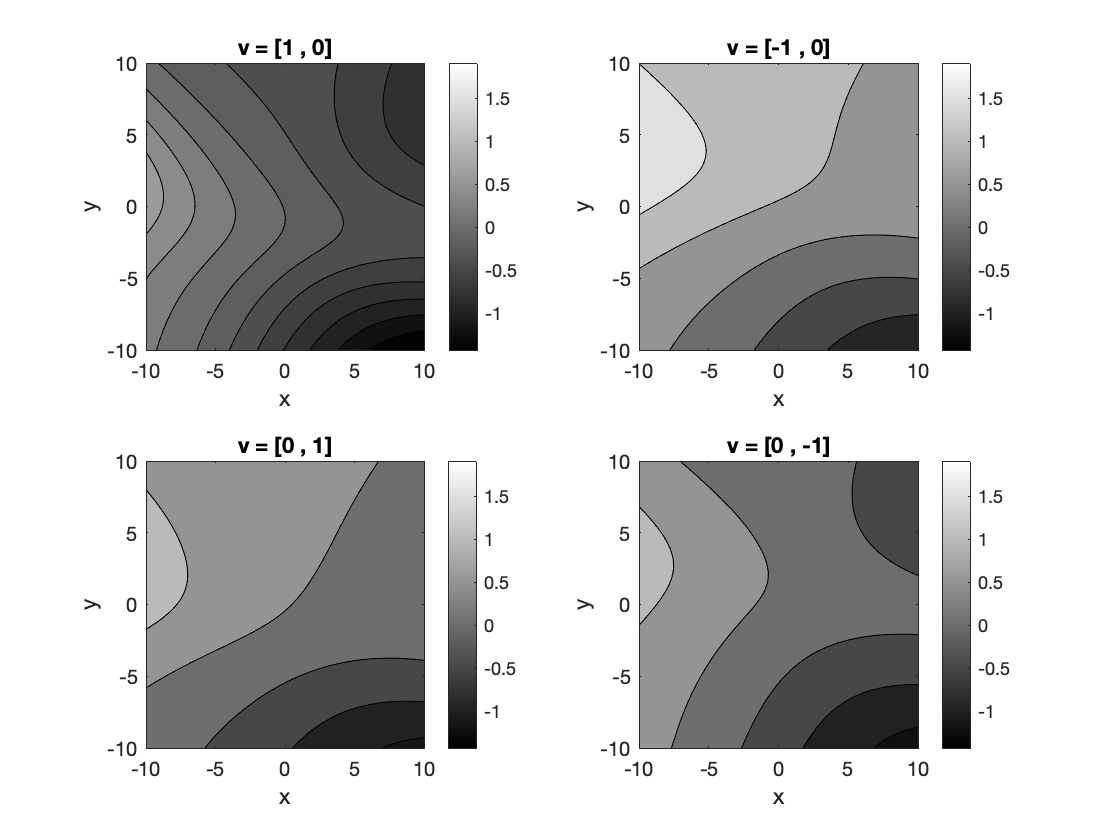}
	\caption{First component of the interaction and wall force for NN approach with optimized parameters.}	
	\label{fig:NN_G1}	
\end{figure}

\begin{figure}
	\includegraphics[scale=0.3]{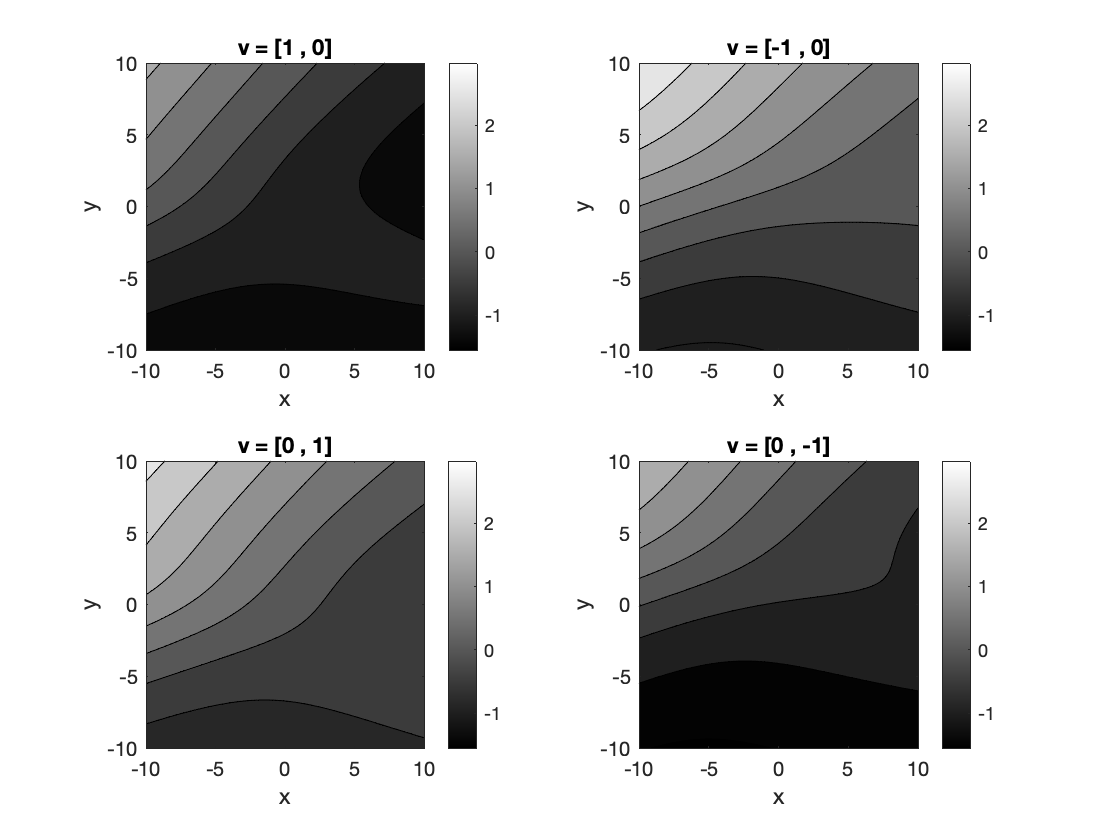}
	\caption{Second component of the interaction and wall force for NN approach with optimized parameters.}	
	\label{fig:NN_G2}	
\end{figure}

The observation that the interaction range is longer in the NN approach motivates to adapt the ranges of the interaction settings. For a similar study as before for the social force model, we use the values given in Table \ref{tab:data6settingsNN}. The arrows in Figure \ref{fig:6NN} show the sum of the interaction and wall forces, i.e., they correspond to Figure \ref{fig:NN_G1} and \ref{fig:NN_G2}. The force vectors in the first and second row coincide. This indicates that the NN model does not have the same problem as forces resulting from radially symmetric potentials. Moreover, the plots indicate that the NN model reacts with evasive behaviour to pedestrians approaching in a longer range (row one and two in Figure \ref{fig:6NN}). In contrast, pedestrians that are close to each other encounter forces in similar directions. Especially, when they head into the same direction (row three in Figure \ref{fig:6NN}). We summarize the findings and comparison of the two models in the conclusion.

\begin{table}
	\begin{tabular}{| c | c | c | c | c | c | c | c | }
		\hline 
		& $x_\text{blue}$ & $x_\text{red}$ & $v_\text{blue}$ & $v_\text{red}$ & $\text{force}_\text{blue}$ & $\text{force}_\text{red}$ \T\B \\
		\hline 
		S1 & $(0;4)$ & $(0;-4)$ & $(0;-1)$ & $(0;1)$ & $(0.1794;0.1517)$ & $(-0.2230;-0.1856)$ \T\B\\
		S2 & $(0;4)$ & $(0;-4)$ & $(0;1)$ & $(0;1)$ & $( 0.4607;0.7512)$ & $( -0.2557;-1.0052)$\T\B\\
		S3 & $(0.01;4)$ & $(-0.01;-4)$ & $(0;-1)$ & $(0;1)$ & $(0.1782;0.1496)$ & $(-0.2218; -0.1858)$ \T\B \\
		S4 & $(0.01;4)$ & $(-0.01;-4)$ & $(0;1)$ & $(0;1)$ & $( 0.4602; 0.7487)$& $(-0.2539;-1.0055)$\T\B\\
		S5 & $(0.5;0)$ & $(-0.5;0)$ & $(0;-1)$ & $(0;1)$ & $(0.1742;-0.7027)$ & $(0.5133;0.8687)$\T\B\\
		S6 & $(0.5;0)$ & $(-0.5,0)$ & $(0,1)$ & $(0,1)$ & $(0.4605;-0.2639)$ & $(0.5612;-0.1540)$ \T\B\\
		\hline 
	\end{tabular}\\
	\vspace*{4mm}
	\caption{Initial data for the study of the interaction and wall forces shown in Figure~\ref{fig:6NN}.}
	\label{tab:data6settingsNN}
\end{table}	

\begin{figure}
	\includegraphics[scale=0.4]{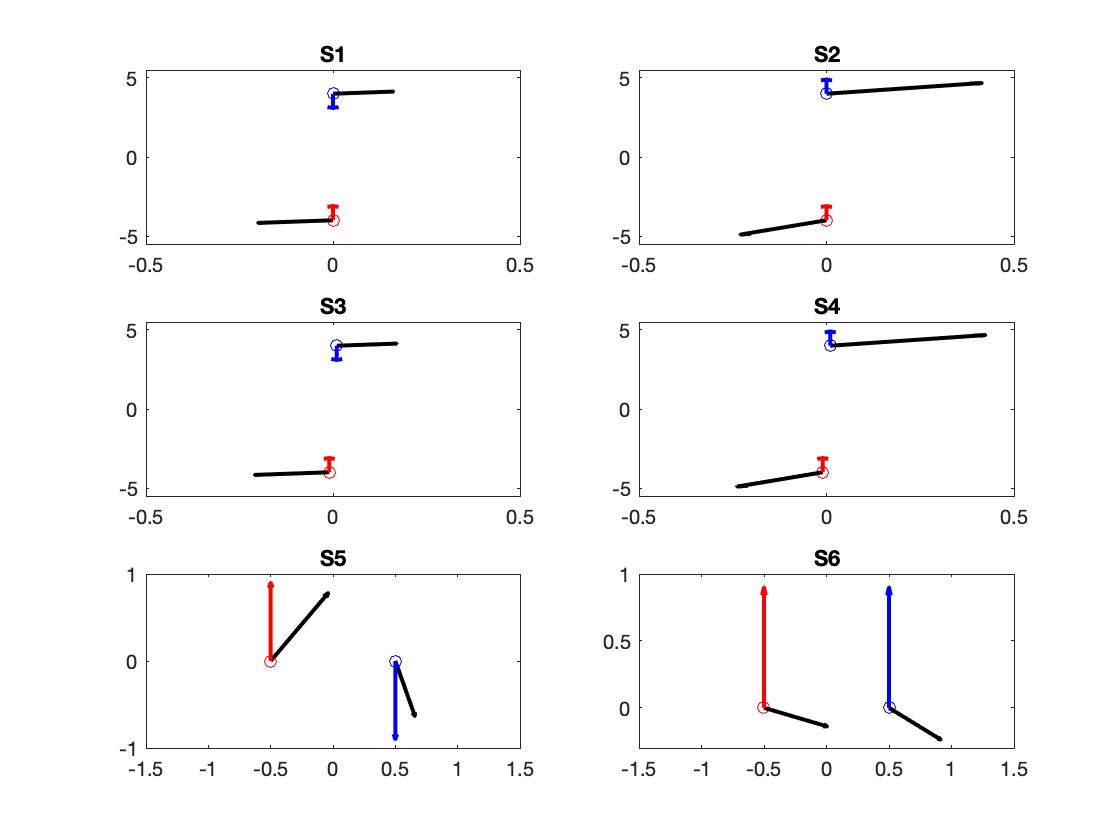}
	\caption{Interaction and wall forces resulting from NN model for different settings, see Table~\ref{tab:data6settingsNN} for more details. The positions of the two interacting pedestrians is marked by the blue and red dot, respectively. Their velocity vectors are the blue and red arrows and the resulting forces are depicted as black arrows. On the left side the pedestrians face each other, on the right side the pedestrians walk in the same direction.}
	\label{fig:6NN}
\end{figure}

\subsection{Comparison of the two approaches}
The numerical results show, that the optimized forces on the social model have a very short range and show expected evasive behaviour for two pedestrians facing each other. On the other hand, the optimized forces of the neural network approach are long ranging. The splitting into a neural network for pairwise interactions and wall interactions is overturned by the optimization. Interestingly, it turns out that a familiar evasive behaviour is recovered when we add the forces resulting from the neural network for pairwise interaction and the neural network modelling the walls. 

The comparison of pairwise interactions show that in case of the social force model, the pedestrians encounter a strong force slowing them down, in particular, the component of force vector pointing in the opposite direction of the velocity vector is very strong. For the neural network, the force vector is almost perpendicular to the velocity vector. Hence, the pedestrians are changing direction but not slowing down as much as in the social force case.
The cost functional values for the two approaches averaged over all samples of the data set are
$$ J_\text{SF} = 5.0841, \qquad J_\text{NN} = 5.5979. $$
We conclude that despite the qualitative differences the two models perform similar in terms of cost. The social force model with short range interaction performs slightly better in this measure as the neural  network approach with long range interaction.

\section{Conclusion}
We proposed to use neural networks to model forces in general interacting particle systems and derived an algorithm to estimate parameters with techniques from optimal control. For validation the algorithm is applied to a traffic data set and a pedestrian data set. The results are compared to the well known LWR-model and social force model, respectively. 

For the traffic data, it turns out that the neural network approach leads to almost linear interaction forces that average the interaction forces resulting for the LWR-model with linear and logarithmic velocity ansatz. The cost functional values are best for the neural  network with 4 neurons in the hidden layer.

In case of the pedestrian dynamics the interaction forces of the social force model and the neural  network approach differ in the range of interaction and the strength. The optimized social force model has strong interaction forces that act on a very short range. In contrast, the forces resulting from the neural  network approach act on a longer range with less strength. Parameter identification with social model performs slightly better than the one based on neural networks in terms of the cost functional value.

Future work includes the investigation of parameter identification problems using neural networks for partial differential equations arising in the context of crowd motion and traffic flow. A performance comparison of the  stochastic gradient descent method to global optimization methods, such as Consensus-based global Optimization \cite{CBO1}, using real data based parameter calibration for interacting particle models is planned as well.

\section*{Acknowledgements}
We thank the project ESIMAS \cite{dataTraffic} and the Institute for Advanced Simulation \cite{PedestrianData} for providing the open source data for traffic and pedestrian dynamics which enabled us to perform the numerical studies of this article.
S.G.~acknowledges support from the DFG within the project GO1920/10-1. 
C.T.~was supported by the European social Fund and by the Ministry Of Science, Research and the Arts Baden-W\"urttemberg.

\bibliographystyle{plain}
\bibliography{biblio,referencesDissSK,BibBookChapter}

\begin{thebibliography}{10}

\bibitem{AlbiPareschi}
G.~Albi and L.~Pareschi.
\newblock Modeling self-organized systems interacting with few individuals:
  from microscopic to macroscopic dynamics.
\newblock {\em Applied Mathematics Letters}, 26(4):397--401, 2013.

\bibitem{Schafe2}
M.~Burger, R.~Pinnau, C.~Totzeck, and O.~Tse.
\newblock Mean-field optimal control and optimality conditions in the space of
  probability measures.
\newblock {\em accepted for publication in SCICON}, 2020.

\bibitem{Schafe1}
M.~Burger, R.~Pinnau, C.~Totzeck, O.~Tse, and A.~Roth.
\newblock Instantaneous control of interacting particle systems in the
  mean-field limit.
\newblock {\em Journal of Computational Physics}, 405:109181, 2020.

\bibitem{CarrilloSwarming}
J.~A. Carrillo, M.~Fornasier, G.~Toscani, and F.~Vecil.
\newblock Particle, kinetic, and hydrodynamic models of swarming.
\newblock In G.~Naldi, L.~Pareschi, and G.~Toscani, editors, {\em Mathematical
  Modeling of Collective Behavior in Socio-Economic and Life Sciences}, pages
  297--336. Birkh{\"a}user Boston, 2010.

\bibitem{Chen2018}
R.~Chen, Y.~Rubanova, J.~Bettencourt, and D.~Duvenaud.
\newblock Neural ordinary differential equations.
\newblock In S.~Bengio, H.~Wallach, H.~Larochelle, K.~Grauman, N.~Cesa-Bianchi,
  and R.~Garnett, editors, {\em Advances in Neural Information Processing
  Systems}, volume~31, pages 6571--6583. Curran Associates, Inc., 2018.

\bibitem{Corbetta2015}
A.~Corbetta, A.~Muntean, and K.~Vafayi.
\newblock Parameter estimation of social forces in pedestrian dynamics models
  via a probabilistic method.
\newblock {\em Mathematical Biosciences and Engineering. MBE}, 12(2):337--356,
  2015.

\bibitem{CrisPicTos2014}
E.~Cristiani, B.~Piccoli, and A.~Tosin.
\newblock {\em Multiscale modeling of pedestrian dynamics}, volume~12 of {\em
  MS\&A. Modeling, Simulation and Applications}.
\newblock Springer, Cham, 2014.

\bibitem{CuckerSmale}
F.~Cucker and S.~Smale.
\newblock Emergent behaviour in flocks.
\newblock {\em IEEE Trans. Automat. Control}, 52:852--862, 2007.

\bibitem{GomesStuartWolfram}
S.N. Gomes, A.M Stuart, and M.-T. Wolfram.
\newblock Parameter estimation for macroscopic pedestrian dynamics models from
  microscopic data.
\newblock {\em SIAM J. Appl. Math.}, 79(4):1475–1500, 2019.

\bibitem{GoettKnapp}
S.~G\"ottlich and S.~Knapp.
\newblock {\em Artificial Neural Networks for the Estimation of Pedestrian
  Interaction Forces}, pages 11--32.
\newblock Springer International Publishing, 2020.

\bibitem{Gruene1}
L.~Gr\"une.
\newblock Computing lyapunov functions using deep neural networks.
\newblock {\em arXiv:2001.08423v3}, 2020.

\bibitem{Gruene2}
L.~Gr\"une.
\newblock Overcoming the curse of dimensionality for approximating lyapunov
  functions with deep neural networks under a small-gain condition.
\newblock {\em arXiv:2001.08423v3}, 2020.

\bibitem{Haber_2017}
E.~Haber and L.~Ruthotto.
\newblock Stable architectures for deep neural networks.
\newblock {\em Inverse Problems}, 34(1):014004, dec 2017.

\bibitem{HelbingMolnar1998}
D.~Helbing and P.~Moln\'{a}r.
\newblock Social force model for pedestrian dynamics.
\newblock {\em Phys. Rev. E}, 51(5):4282--4286, 1995.

\bibitem{HUUP}
M.~Hinze, R.~Pinnau, M.~Ulbrich, and S.~Ulbrich.
\newblock {\em Optimization with PDE Constraints}.
\newblock Springer, 2009.

\bibitem{LWR_HR}
H.~Holden and N.~H. Risebro.
\newblock Follow-the-leader models can be viewed as a numerical approximation
  to the lighthill-whitham-richards model for traffic flow.
\newblock {\em Netw. Heterog. Media}, 13:409--421, 2018.

\bibitem{dataTraffic}
E.~Kallo, A.~Fazekas, S.~Lamberty, and M.~Oeser.
\newblock Microscopic traffic data obtained from videos recorded on a german
  motorway.
\newblock {\em Mendeley Data}, V1, 2019.

\bibitem{Energy}
A.P. Marugán, F.P. García~Márquez, J.M. Pinar~Perez, and D.~Ruiz-Hernández.
\newblock A survey of artificial neural network in wind energysystems.
\newblock {\em Appl. Energy}, 228:1822--1836, 2018.

\bibitem{PareschiToscani2013}
L.~Pareschi and G.~Toscani.
\newblock {\em Interacting Multiagent Systems: Kinetic equations and Monte
  Carlo methods}.
\newblock Oxford University Press, 2013.

\bibitem{CBO1}
R.~Pinnau, C.~Totzeck, O.~Tse, and S.~Martin.
\newblock A consensus-based model for global optimization and its mean-field
  limit.
\newblock {\em Mathematical Models and Methods in Applied Sciences}, 27(1),
  2017.

\bibitem{Brakes}
V.~Ricciardi, K.~Augsburg, S.~Gramstat, V.~Schreiber, and V.~Ivanov.
\newblock Survey on modelling and techniques for frictionestimation in
  automotive brakes.
\newblock {\em Appl. Sci.}, 7(9):873, 2017.

\bibitem{PedestrianData}
A.~Seyfried and M.~Boltes.
\newblock Data archive of experimental data from studies about pedestrian
  dynamics.
\newblock {\em DOI: 10.34735/ped.2013.5}, last visited: Jan 21, 2021.

\bibitem{Tordeux2019}
A.~Tordeux, M.~Chraibi, A.~Seyfried, and A.~Schadschneider.
\newblock Prediction of pedestrian speed with artificial neural networks.
\newblock In S.~H. Hamdar, editor, {\em Traffic and Granular Flow '17}, pages
  327--335, Cham, 2019. Springer International Publishing.

\bibitem{toscani2006kinetic}
G.~Toscani.
\newblock Kinetic models of opinion formation.
\newblock {\em Communications in mathematical sciences}, 4(3):481--496, 2006.

\bibitem{CollisionAvoidance}
C.~Totzeck.
\newblock An anisotropic interaction model with collision avoidance.
\newblock {\em Kinetic and related models}, 13(6):1219--1242, 2020.

\end{thebibliography}

\end{document}